\pdfoutput=1
\documentclass[a4paper,USenglish,cleveref, autoref, thm-restate,nolineno]{socg-lipics-v2021}
\hideLIPIcs
\title{Sublevels in arrangements and the spherical arc crossing number of complete graphs} %TODO Please add

\author{Elizaveta Streltsova}{ISTA, Klosterneuburg, Austria}{estrelts@ist.ac.at
}{}{}%TODO mandatory, please use full name; only 1 author per \author macro; first two parameters are mandatory, other parameters can be empty. Please provide at least the name of the affiliation and the country. The full address is optional. Use additional curly braces to indicate the correct name splitting when the last name consists of multiple name parts.

\author{Uli Wagner}{ISTA, Klosterneuburg, Austria \and \url{https://ist.ac.at/en/research/wagner-group/}}{uli@ist.ac.at}{https://orcid.org/0000-0002-1494-0568}{}

\authorrunning{E.\ Streltsova and U.\ Wagner} %TODO mandatory. First: Use abbreviated first/middle names. Second (only in severe cases): Use first author plus 'et al.'

 %TODO mandatory: Please choose ACM 2012 classifications from https://dl.acm.org/ccs/ccs_flat.cfm

 %TODO mandatory; please add comma-separated list of keywords

\category{} %optional, e.g. invited paper

\relatedversion{} %optional, e.g. full version hosted on arXiv, HAL, or other respository/website
\nolinenumbers %uncomment to disable line numbering

\EventEditors{Oswin Aichholzer and Haitao Wang}
\EventNoEds{2}
\EventLongTitle{41st International Symposium on Computational Geometry (SoCG 2025)}
\EventShortTitle{SoCG 2025}
\EventAcronym{SoCG}
\EventYear{2025}
\EventDate{June 23--27, 2025}
\EventLocation{Kanazawa, Japan}
\EventLogo{socg-logo.pdf}
\SeriesVolume{332}
\ArticleNo{XX} 
\newcounter{sideremark}

\newcommand{\arr}{\mathcal{A}}

\DeclareMathOperator{\sgn}{sgn}
\DeclareMathOperator{\conv}{conv}

\DeclareMathOperator{\cross}{cr}

\renewcommand{\leq}{\leqslant}

\renewcommand{\geq}{\geqslant}

\renewcommand{\phi}{\varphi}

\newcommand{\R}{\mathbb{R}}
\newcommand{\Z}{\mathbb{Z}}

\usepackage{mathtools}
\DeclarePairedDelimiter\ceil{\lceil}{\rceil}
\DeclarePairedDelimiter\floor{\lfloor}{\rfloor}

\newcommand{\cyclic}{V_{\textup{cyclic}}}
\newcommand{\cocyclic}{V_{\textup{cocyclic}}}
\usepackage{appendix}
\usepackage[normalem]{ulem}
\usepackage{hyperref}

\begin{document}

\maketitle
\thispagestyle{empty}

\begin{abstract}
Levels and sublevels in arrangements---and, dually, $k$-sets and $(\leq k)$-sets---are fundamental notions in discrete and computational geometry and natural generalizations of convex polytopes, which  correspond to the $0$-level. A long-standing conjecture of Eckhoff, Linhart, and Welzl, which would generalize McMullen's Upper Bound Theorem for polytopes and provide an exact refinement of asymptotic bounds by Clarkson, asserts that for all $k\leq \floor{\frac{n-d-2}{2}}$, the number of $(\leq k)$-sets of a set $S$ of $n$ points in $\mathbf{R}^d$ is maximized if $S$ is the vertex set of a neighborly polytope.

As a new tool for studying this conjecture and related problems, we introduce the \emph{$g$-matrix}, which generalizes both the $g$-vector of a simple polytope and a Gale dual version of the $g$-vector studied by Lee and Welzl. Our main result is that the $g$-matrix of every vector configuration in $\mathbf{R}^3$ is non-negative. This implies the following result (which, by Gale duality, is equivalent to the case $d=n-4$ of the Eckhoff--Linhart--Welzl conjecture). 
Let $V=\{v_1,\dots,v_n\} \subset \mathbf{R}^3$ be a configuration 
of $n$ vectors in general position and consider the sign patterns $(\sgn(\lambda_1),\dots,\sgn(\lambda_n))\in \{-1,0,+1\}^n$ of nontrivial linear dependencies $\sum_{i=1}^n \lambda_i v_i=0$.
For $0\leq t \leq s \leq n$, let $f^*_{s,t}(V)$ denote the number of dependency patterns with $t$ negative signs $-1$ and $s$ non-zero signs $\pm 1$. Then, for all $s\leq n$, the numbers $f^*_{s,0}(V)$ and $f^*_{s,\leq 1}(V):=f^*_{s,0}(V)+f^*_{s,1}(V)$ are maximized if $V$ is \emph{coneighborly}, i.e., if every open linear halfspace contains at least $\lfloor (n-2)/2 \rfloor$ of the vectors in $V$.

As a corollary, we obtain the following result about crossing numbers: Let us normalize the vectors in $V$ to unit length (which does not affect the dependency patterns) and connect every pair of vectors by the unique shortest geodesic arc between them in the unit sphere $S^2$. This yields a drawing of the complete graph $K_n$ in $S^2$, which we call a \emph{spherical arc drawing}. These drawings generalize the well-studied class of \emph{rectilinear drawings} of $K_n$ ($n$ points in general position in $\R^2$ connected by straight-line segments), which correspond to spherical arc drawings for which $V$ is contained in an open hemisphere. Complementing previous results for rectilinear drawings, our result implies that the number of crossings in any spherical arc drawing of $K_n$ is at least $\frac{1}{4}\lfloor \frac{n}{2}\rfloor \lfloor \frac{n-1}{2}\rfloor \lfloor \frac{n-2}{2}\rfloor \lfloor \frac{n-3}{2}\rfloor$, which equals the conjectured value of the crossing number of $K_n$. Moreover, unlike for rectilinear drawings, the lower bound for spherical arc drawings is attained, by coneighborly configurations.
\end{abstract}

\section{Introduction}
\label{s:intro}
In this paper, we study the interplay between several classical topics in discrete and computational geometry: the combinatorial theory of convex polytopes, the complexity of $(\leq k)$-sublevels in arrangements, Radon partitions, and crossing numbers.

Let $V = \{ v_1, \ldots, v_n \} \subset \R^3$ be a set of $n\geq 3$ vectors; fixing the labeling of the vectors, we will also view $V$ as an $(3 \times n)$-matrix $V=[v_1|\dots|v_n] \in \R^{3 \times n}$ with column vectors $v_i$. Unless stated otherwise, we assume that $V$ is in general position, i.e., that any $3$ of the vectors are linearly independent. We refer to $V$ as a \emph{vector configration} of $\emph{rank}$ $3$. 

Let $S^2$ be the unit sphere in $\R^3$, let $\langle\cdot ,\cdot \rangle$ denote the standard inner product in $\R^3$, and let $\sgn(x)\in \{-1,0,+1\}$ denote the sign of a real number $x\in \R$. 
For $F\in \{-1,0,+1\}^n$, let $F_+$, $F_0$, and $F_-$ denote the subsets of indices $i\in [n]$ such that $F_i=+1$, $F_i=0$, and $F_i=-1$.

\subsection{Dependency and Dissection Patterns}
\label{sec:DepDiss}
We consider two sets $\mathcal{F}(V),\mathcal{F}^*(V) \subset \{-1,0,+1\}^n$ of sign vectors defined as follows. 

By definition, $\mathcal{F}^*(V)$ is the set of sign vectors $(\sgn(\lambda_1),\dots,\sgn(\lambda_n)) \in \{-1,0,+1\}^n$ given by non-trivial linear dependencies $\sum_{i=1}^n \lambda_i v_i=0$ (with coefficients $\lambda_i\in \R$, not all of them zero). We call $\mathcal{F}^*(V)$ the \emph{dependency patterns} of $V$.

We define $\mathcal{F}(V)$ as the set of all sign vectors $(\sgn(\langle v_1,u\rangle),\dots,\sgn(\langle v_n,u\rangle)) \in \{-1,0,+1\}^n$, where $u$ ranges over all non-zero vectors in $\R^3$ (equivalently, unit vectors in $S^2$). Thus, the elements $\mathcal{F}(V)$ encode partitions $V=V_-\sqcup V_0 \sqcup V_+$ by oriented planes through the origin, and we call them the \emph{dissection patterns} of $V$. Equivalently, every vector $v_i \in V$ defines a great circle 
$
H_i = \{ x \in S^2 \mid \langle v_i, x \rangle = 0 \}
$
in $S^2$ and two open hemispheres $H_i^+ = \{ x \in S^2 \mid \langle v_i, x \rangle >0 \}$ and $H_i^- = \{ x \in S^2 \mid \langle v_i, x \rangle < 0 \}.$
The resulting \emph{arrangement} $\arr(V)=\{H_1^+,\dots,H_n^+\}$ of hemispheres in $S^2$ determines a decomposition of $S^2$ into \emph{faces} of dimensions $0,1,2$, 
where two points $u,u'\in S^2$ lie in the relative interior of the same face iff $\sgn(\langle v_i,u \rangle)=\sgn(\langle v_i,u'\rangle)$ for $1\leq i\leq n$. Thus, we can identify each face of $\arr(V)$ with its \emph{signature} $F\in \mathcal{F}(V)$; by general position, the face with signature $F$ has dimension $2-|F_0|$ (there are no faces with $|F_0|>2$, i.e., the arrangement is \emph{simple}). Moreover, we call $|F_-|$ the \emph{level} of the face. We 
refer to the correspondence between $V$ and $\mathcal{A}(V)$ as \emph{polar duality} (to distinguish it from \emph{Gale duality}, see Sec.~\ref{sec:Gale} below).

Both $\mathcal{F}(V)$ and $\mathcal{F}^*(V)$ are invariant under invertible linear transformations of $\R^r$ and under positive rescaling (multiplying each vector $v_i$ by some positive scalar $\alpha_i>0$).

\begin{definition}[$f$ and $f^*$]
\label{def:f-poly}
For integers $s$ and $t$, define%
\footnote{By general position, $f_{s,t}(V)=0$ unless $0\leq s\leq d$ and $0\leq t\leq n-s$, and 
$f^*_{s,t}(V)=0$ unless $4\leq s \leq n$ and $0\leq t\leq s$. However, it will occasionally be convenient to allow an unrestricted range of indices.}
$$
f_{s,t}(V):=|\{F\in \mathcal{F}(V)\mid |F_0|=s,|F_-|=t\}|, \qquad f^*_{s,t}(V):=|\{F\in \mathcal{F}^*(V)\mid |F_-|=t, |F_+|=s-t\}|
$$
Thus, $f_{s,t}(V)$ counts the $(2-s)$-dimensional faces of level $t$ in $\arr(V)$, and $f^*_{s,t}(V)$ counts the dependency patterns with $t$ negative entries $-1$ and $s$ non-zero entries $\pm 1$.

Together, these numbers form the \emph{$f$-matrix} $f(V)=[f_{s,t}(V)]$ and the \emph{$f^*$-matrix} $f^*(V)=[f^*_{s,t}(V)]$. Equivalently, we can encode this data into two bivariate polynomials
$f_V(x,y)$ and $f^*_V(x,y)$ in $\Z[x,y]$, the \emph{$f$-polynomial} and the \emph{$f^*$-polynomial} of $V$, which are defined by 
\begin{align*} 
f_{V}(x,y) & :=  \sum_{F\in \mathcal{F}(V)} x^{|F_0|} y^{|F_-|} = \sum_{s,t} f_{s,t}(V)\, x^s y^t,\\
f^*_{V}(x,y) &:= \sum_{F\in \mathcal{F}^*(V)} x^{|F_0|} y^{|F_-|} = \sum_{s,t} f^*_{s,t}(V)\, x^{n-s} y^t.
\end{align*}
\end{definition}
We are ready to state our first result. We say that $V = \{ v_1, \ldots, v_n \} \subset \R^3$ is \emph{coneighborly} if every open linear halfspace (bounded by a plane through the orgin) in $\R^3$ contains at least $\lfloor (n-2)/2 \rfloor$ of the vectors in $V$. Coneighborly configurations in $\R^3$ are precisely the \emph{Gale duals} of $(n-4)$-dimensional \emph{neighborly polytopes} with $n$ vertices (see Section~\ref{sec:Gale}, and Example~\ref{ex:cyclic-cocyclic} for a specific example, \emph{cocyclic} configurations). It is known \cite{Padrol:2013aa} that there are at least $n^{\frac{3}{2}(1-o(1))n}$ combinatorially distinct $(n-4)$-dimensional neighborly polytopes with $n$ vertices, hence at least that many different combinatorial types of 
coneighborly configurations of $n$ vectors in $S^2$. 
\begin{theorem} 
\label{thm:GUBT-fstar}
Let $V = \{ v_1, \ldots, v_n \} \subset \R^3$ be a vector configuration in general position. Then, for all $s\leq n$, the numbers $f^*_{s,0}(V)$ and $f^*_{s,\leq 1}(V):=f^*_{s,0}(V)+f^*_{s,1}(V)$ are maximized if $V$ is coneighborly.
\end{theorem}
The result for $f^*_{s,0}(V)$ (and its generalization to vector configurations of arbitrary rank $r$) is known: It is equivalent, by Gale duality, to McMullen's Upper Bound Theorem for polytopes \cite{McMullen:1970aa}; a direct proof was given by Welzl \cite{Welzl:2001aa}. The result for $f^*_{s,\leq 1}(V)$ is new and confirms, for $d=n-4$, a generalization of the Upper Bound Theorem for sublevels of arrangements in $S^d$ conjectured by Eckhoff~\cite{Eckhoff:1993aa}, Linhart~\cite{Linhart:1994aa}, and Welzl~\cite{Welzl:2001aa} (see Sec.~\ref{sec:Gale}, Conjecture~\ref{GUBC}). 

\subsection{The Spherical Arc Crossing Number of $K_n$}
\label{subsec:crossing}
As an application of Theorem~\ref{thm:GUBT-fstar}, we obtain a result about crossing numbers. Determining the crossing number $\cross(K_n)$ of the complete graph $K_n$ (the minimum number of crossings in any drawing of $K_n$ in the plane $\R^2$, or equivalently in the sphere $S^2$, with edges represented by arbitrary Jordan arcs) is one of the foundational unsolved problems in geometric graph theory. This problem was first studied by Hill in the 1950's, who conjectured the following: 
\begin{conjecture}[Hill]
\begin{equation}
\label{conj:Hill}
\cross(K_n)=X(n):=\frac{1}{4}\Big\lfloor \frac{n}{2}\Big\rfloor \Big\lfloor \frac{n-1}{2} \Big\rfloor \Big\lfloor \frac{n-2}{2} \Big\rfloor \Big\lfloor \frac{n-3}{2}\Big\rfloor =\frac{3}{8}\binom{n}{4}+O(n^3)
\end{equation}
\end{conjecture}
This is known to hold for $n\leq 12$, but remains open in general (see \cite[Sec.~1.3]{Schaefer:2018aa} or \cite{Szekely:2016aa} for further background and references). There are several families of drawings showing that $\cross(K_n)\leq X(n)$ for all $n$, but the best lower bound to date is %roughly 
$\cross(K_n) \geq 0.985 \cdot X(n)$  \cite{Balogh:2019aa}.

Here, we prove Hill's conjecture for the following class of drawings. Let $V=\{v_1,\dots,v_n\} \subset S^2$ be a configuration of $n$ \emph{unit} vectors in general position. If we connect every pair of vectors in $V$ by the shortest geodesic arc between them in $S^2$ (which is unique, since no two vectors are antipodal, by general position) we obtain a drawing of the complete graph $K_n$ in $S^2$, which we call a \emph{spherical arc drawing}. Let 
$\cross(V)$ 
denote the number of crossings in this drawing.
\begin{theorem} 
\label{thm:crossings}
For every configuration of $n$ vectors in general position in $S^2$, 
\[\cross(V) \geq \frac{1}{4}\Big\lfloor \frac{n}{2}\Big\rfloor \Big\lfloor \frac{n-1}{2} \Big\rfloor \Big\lfloor \frac{n-2}{2} \Big\rfloor \Big\lfloor \frac{n-3}{2}\Big\rfloor
\]
Moreover, the lower bound is attained with equality if $V$ is coneighborly.
\end{theorem}

The fact that coneighborly configurations yield spherical arc drawings of $K_n$ achieving the number $X(n)$ of crossings in Hill's conjecture and the connection to the Eckhoff--Linhart--Welzl conjecture were first observed by Wagner \cite{Wagner:2006aa,Wagner:2008aa}. For a connection to geometric probability, in particular to old results of Wendel~\cite{Wendel:1962aa}  and Moon~\cite{Moon:1965aa}, see Remark~\ref{rem:Wendel}.

Spherical arc drawings generalize the well-studied class of \emph{rectilinear drawings} of $K_n$, given by $n$ points in general position in $\R^2$ connected by straight-line segments; more precisely, rectilinear drawings correspond to the sub-class of sperical arc drawings for which the vector configuration $V$ is \emph{pointed}, i.e., contained in an open linear halfspace. Pointed configurations $V\subset \R^3$ correspond to point sets $S\subset \R^2$ by radial projection,\footnote{Given $S=\{p_1,\dots,p_n\} \subset \R^2$, we obtain a pointed vector configuration $V=\{v_1,\dots,v_n\} \subset \R^3$ by setting $v_i:=(1,p_i)\in \{1\}\times \R^d \subset \R^3$. Conversely, if $V \subset \{x \in \R^3\colon \langle u,x\rangle>0\}$ for some $u\in S^2$ then we radially project each vector $v_i$ to the point $p_i:=\frac{1}{\langle u, v_i\rangle} v_i$ in the tangent plane $\{x\in \R^r\colon \langle u,x\rangle = 1\} \cong \R^2$.}
and under this correspondence, $\mathcal{F}(V)$ and $\mathcal{F}^*(V)$ encode the partitions of $S$ by affine lines and \emph{Radon partitions} $S=S_-\sqcup S_0 \sqcup S_+$, $\conv(S_+) \cap \conv(S_-)\neq \emptyset$, respectively. Moreover, for a pointed $V \subset S^2$, spherical arcs correspond to straight-line segments.  

Theorem~\ref{thm:crossings} complements earlier results of Lov\'asz, Vesztergombi, Wagner, and Welzl~\cite{Lovasz:2004aa} and \'Abrego and Fern\'andez-Merchant~\cite{Abrego:2005aa}, who showed that the \emph{rectilinear crossing number} $\overline{\cross}(K_n)$ (the minimum number of crossings in any rectilinear drawing of $K_n$) is at least $X(n)$; in fact, $\overline{\cross}(K_n)\geq (\frac{3}{8}+\varepsilon+o(1))\binom{n}{4}$ for some constant $\varepsilon>0$ \cite{Lovasz:2004aa}. Thus, unlike the spherical arc crossing number, the rectilinear crossing number $\overline{\cross}(K_n)$ is strictly larger than $X(n)$ (and hence larger than $\cross(K_n)$) in the asymptotically leading term. We refer to  \cite{Abrego:2013ab} for a detailed survey, including a series of subsequent improvements \cite{Balogh:2006aa,Aichholzer:2007aa,Abrego:2008aa,Abrego:2008ab} leading to the currently best bound \cite{Abrego:2012aa} $\overline{\cross}(K_n) > 277/729 \binom{n}{4}+O(n^3) > 0.37997 \binom{n}{4}+O(n^3)$. We remark that the arguments in~\cite{Lovasz:2004aa,Abrego:2005aa} have been generalized to verify Hill's conjecture for other classes of drawings, including $2$-page drawings \cite{Abrego:2013aa}, monotone drawings \cite{Balko:2015aa}, cylindrical, $x$-bounded, and shellable drawings  \cite{Abrego:2014aa}, bishellable drawings \cite{Abrego:2018aa}, and seq-shellable drawings \cite{Mutzel:2018aa}. Currently we do not know %at the moment 
how spherical arc drawings relate to these other classes of drawings.

\subsection{The $g$-Matrix}
\label{sec:g-matrix-intro}
The central new notion of this paper is the $g$-matrix $g(V\to W)$ of a pair of vector configurations, which encodes the differences $f(W)-f(V)$ and $f^*(W)-f^*(V)$ of $f$-matrices and $f^*$-matrices and which generalizes both the classical $g$-vector of a simple polytope and a Gale dual version of the $g$-vector studied by Lee~\cite{Lee:1991aa} and Welzl~\cite{Welzl:2001aa}. 

The definition of the $g$-matrix (which we present here for vector configurations in $\R^3$ and which easily generalizes to higher dimensions, see \cite{SW:DS}) is based on how $\mathcal{F}$ and $\mathcal{F}^*$ (and hence $f$ and $f^*$) change by \emph{mutations} during a \emph{continuous motion} (this idea has a long history in discrete geometry, including in the study of dissection patterns, see, e.g., \cite{Andrzejak:1998aa}).

During a mutation, a unique triple of vectors, indexed by some $R=\{i_1,i_2,i_3\}\subset [n]$, become linearly dependent, the orientation of this triple changes, and all other triples remain linearly independent. Let $V$ and $W$ denote the vector configurations before and after the mutation. In terms of the polar dual arrangements in $S^2$, the three great circles indexed by $R$ intersect in a pair of antipodal points $u,-u\in S^2$ during the mutation. Immediately before and immediately after the mutation, these three great circles bound an antipodal pair of small triangular faces $\sigma,-\sigma$ in $\arr(V)$ and a corresponding pair $\tau,-\tau$ of triangular faces in $\arr(W)$, respectively, see Figures~\ref{fig:0k3k} and \ref{fig:1k2k}. We say that $\sigma$ and $-\sigma$ \emph{disappear} during the mutation, and that $\tau$ and $-\tau$ \emph{appear}. Let $Y \in \mathcal{F}(W)$ be the signature of $\tau$. Set $j:=|R\cap Y_-|$ and $k:=|([n]\setminus R)\cap Y_-|$; we call $(j,k)$ the \emph{type} of $\tau$. It is easy to see that $\sigma$ has type $(3-j,k)$. Analogously, $-\tau$ and $-\sigma$ are of type $(3-j,n-3-k)$ and $(j,n-3-k)$, respectively. 
\begin{figure}[ht]
\begin{center}
\includegraphics[scale=0.7]{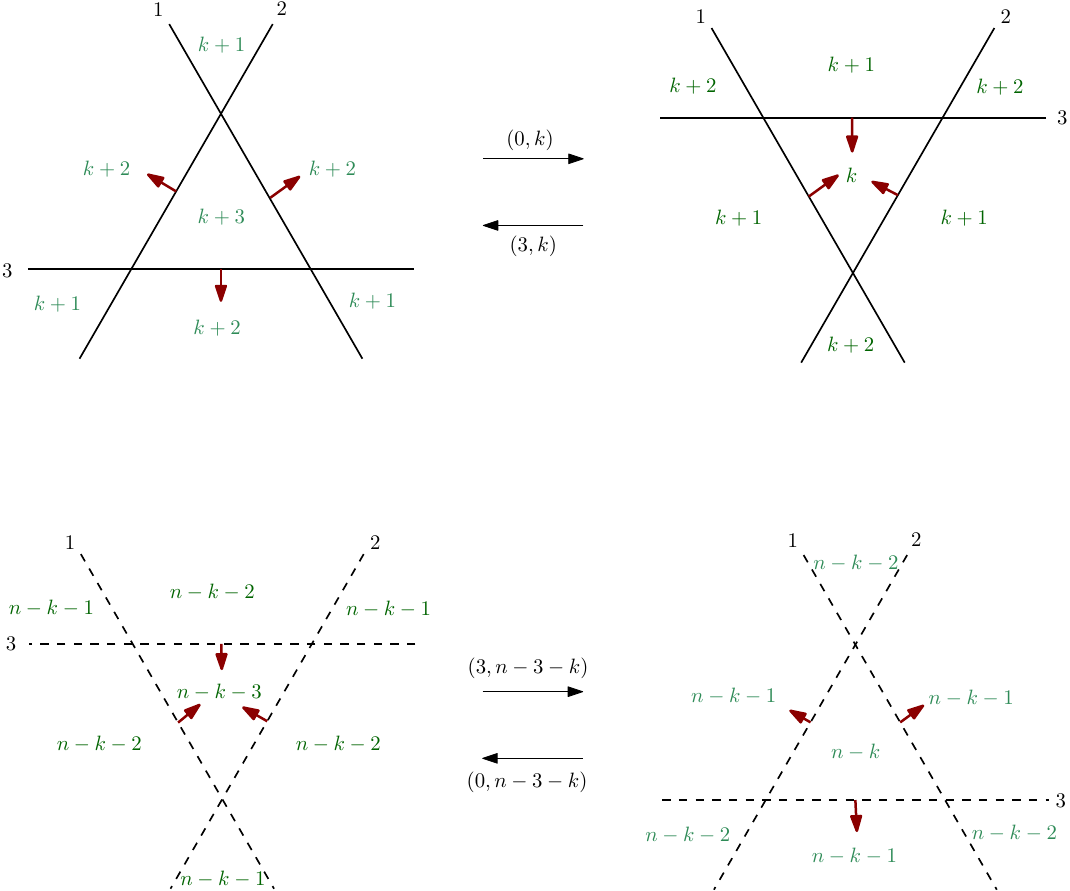}
\end{center}
\caption{A mutation between an antipodal pair of triangles $\sigma,-\sigma$ of types $(3,k)$ and $(0,n-3-k)$ (left column top and bottom) and an antipodal pair $\tau,-\tau$ of triangles of types $(0,k)$ and $(3,n-3-k)$ (right column top and bottom). The horizontal arrows are marked with the types of the appearing triangles as we move between the left and the right column. The labels in the full-dimensional cells indicate their levels, and the little arrows indicate positive hemispheres.}
\label{fig:0k3k}
\end{figure}

\begin{figure}[ht]
\begin{center}
\includegraphics[scale=0.7]{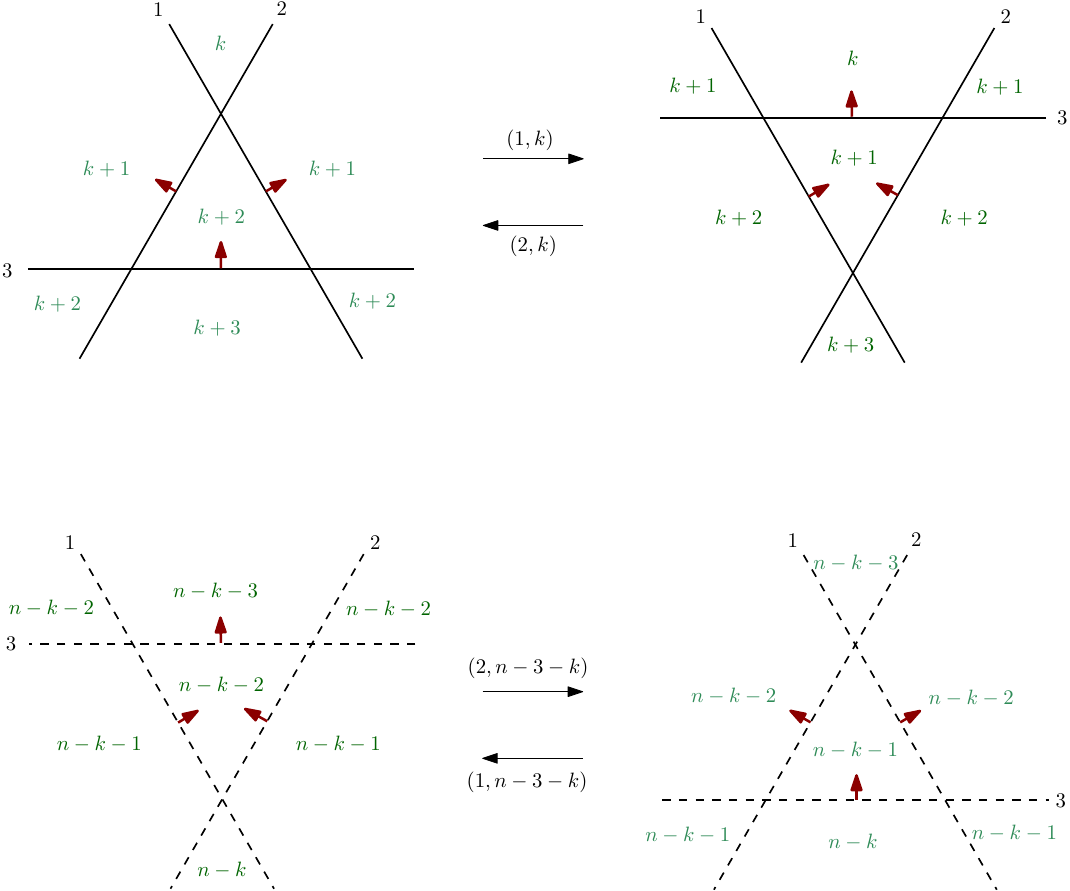}
\end{center}
\caption{A mutation between an antipodal pair of triangles $\sigma,-\sigma$ of types $(2,k)$ and $(1,n-3-k)$ (left column) and an antipodal pair $\tau,-\tau$ of triangles of types $(1,k)$ and $(2,n-3-k)$ (right column).
}
\label{fig:1k2k}
\end{figure}
Any two configurations $V=\{v_1,\dots,v_n\}$ and $W=\{w_1,\dots,w_n\}$ of $n$ vectors in general position in $\R^3$ are connected by a finite sequence of mutations. We define $g_{j,k}(V\to W)$ as the \emph{net} number of triangular faces of type $(j,k)$ that appear during this sequence (every mutation contributes $+1$, $-1$, or $0$), see Definition~\ref{def:g-pair} for the details; this depends only on $V$ and $W$, not on the sequence of mutations, see Remark~\ref{rem:g-f+DS-g}.
Analyzing how the $f$-polynomial and the $f^*$-polynomial change during a mutation (Lemmas~\ref{lem:f-mutation} and \ref{lem:fstar-mutation}) %; this 
yields the following:

\begin{theorem} 
\label{thm:f-g}
Let $V,W \in \R^{3\times n}$ be a pair of vector configurations in general position. 

The \emph{$g$-matrix} $g(V\to W)$ of the pair is an $4\times (n-4)$-matrix  with integer entries $g_{j,k}:=g_{j,k}(V\to W)$, $0\leq j\leq 3$, $0\leq k \leq n-3$, which has the following properties:  
\begin{enumerate}
\item $g(W\to V)=-g(V\to W)$. If $U\in \R^{r\times n}$, then $g(U\to W)=g(U\to V)+g(V\to W)$. 
\item For $0\leq j\leq 3$ and $0\leq k\leq n-3$, the $g$-matrix satifies the skew-symmetries
\begin{equation}
\label{eq:g-skew}
g_{j,k}=-g_{3-j,k}=-g_{j,n-3-k}=g_{3-j,n-3-k}
\end{equation}
Thus, the $g$-matrix is determined by the submatrix $[g_{j,k}\colon 0\leq j \leq 1, 0\leq k \leq \floor{\frac{n-4}{2}}]$, which we call the \emph{small $g$-matrix}. 
Equivalently, the \emph{$g$-polynomial} $g(x,y):=g_{V\to W}(x,y):=\sum_{j,k} g_{j,k} x^j y^k \in \Z[x,y]$ satisfies
\begin{equation}
\label{eq:g-poly-skew}
\textstyle g(x,y)=-x^3g(\frac{1}{x},y)=-y^{n-3}g(x,\frac{1}{y})=x^3y^{n-3}g(\frac{1}{x},\frac{1}{y})
\end{equation}
\item The $g$-polynomial determines the difference $f(x,y):=f_W(x,y)-f_V(x,y)$ by
\begin{equation}
\label{eq:f-poly-g-poly}
f(x,y) = (1+x)^3 g (\textstyle \frac{x+y}{1+x},y)=\sum_{j=0}^3 \sum_{k=0}^{n-3} g_{j,k}\cdot (x+y)^j (1+x)^{3-j} y^k
\end{equation}
\item The $g$-polynomial determines the difference $f^*(x,y)=f^*_W(x,y)-f^*_V(x,y)$ by
\begin{equation}
\label{eq:f*-g*}
\textstyle f^*(x,y) = \sum_{j,k} \underbrace{g_{j,k}(W \to V)}_{=-g_{j,k}(V \to W)} (x+y)^k (x+1)^{n-3-k} y^j = -(x+1)^{n-3} g(y,\frac{x+y}{x+1})
\end{equation}

\end{enumerate}
\end{theorem}
By comparing coefficients in \eqref{eq:f-poly-g-poly} and \eqref{eq:f*-g*}, we get:
\begin{corollary} 
\label{cor:f-g-fstar-gstar-coefficients}
For $0\leq s\leq 2$ and $0\leq t\leq n$,
\begin{equation}
\label{eq:f-g-coefficients}
f_{s,t}(W)-f_{s,t}(V) = \sum_{j,k} \binom{j}{t-k}\binom{3-j}{s-j+t-k} g_{j,k}(V\to W),
\end{equation}
Moreover, for $4\leq s \leq n$ and $0\leq t \leq s$,
\begin{equation}
\label{eq:fstar-g-coefficients}
f^*_{s,t}(V)-f^*_{s,t}(W) = \sum_{j,k} \binom{k}{t-j}\binom{n-3-k}{s-t+j-3}g_{j,k}(V\to W)
\end{equation}
In particular (by specializing \eqref{eq:fstar-g-coefficients} to $t=0,1$), we get that for $4\leq s \leq n$
\begin{equation}
\label{eq:fstar0-g0-coefficients}
f^*_{s,0}(V)-f^*_{s,0}(W) = \sum_{k} \binom{n-3-k}{s-3}g_{0,k}(V\to W)
\end{equation}
and
\begin{equation}
\label{eq:fstar-g1-coefficients}
\textstyle f^*_{1,0}(V)-f^*_{1,0}(W) = \sum_{k} \binom{n-3-k}{s-3}g_{1,k}(V\to W) + \sum_{k}k\binom{n-3-k}{s-4}g_{0,k}(V\to W)
\end{equation}
\end{corollary}
\begin{remark} 
\label{rem:g-f+DS-g}
It is easy to show that the system of equations \eqref{eq:fstar0-g0-coefficients} can be inverted, hence the differences $f^*_{s,0}(V)-f^*_{s,0}(W)$, $4\leq s\leq n$, determine $g_{0,k}(V\to W)$, $0\leq k\leq n-3$; by the same token, the numbers $g_{1,k}(V\to W)$ are then determined by $f^*_{1,0}(V)-f^*_{1,0}(W)$, $4\leq s\leq n$, together with the already determined $g_{0,k}(V\to W)$, $0\leq k\leq n-3$. By skew-symmetry, this determines the entire $g$-matrix $g(V\to W)$, which thus only depends on the difference $f^*(V)-f^*(W)$ of $f^*$-matrices. Analogously, the $g$-matrix is determined by the difference of $f$-matrices. Thus, Theorem~\ref{thm:f-g} could be taken as a formal definition of the $g$-matrix.
Moreover, the skew-symmetry $g(x,y)=-x^r g(\frac{1}{x},y)$ reflects the Dehn--Sommerville relation \eqref{eq:DS-f-poly}, and the symmetry $g(x,y)=x^ry^{n-r}g(\frac{1}{x},\frac{1}{y})$ reflects the antipodal symmetry \eqref{eq:f-symm}.
\end{remark}

We will need the following two special configurations: 
\begin{example}[{Cyclic and Cocyclic Configurations}]
\label{ex:cyclic-cocyclic}
Let $t_1<t_2<\dots<t_n$ be real numbers and define $v_i:= (1,t_i,t_i^2) \in \R^3$. We call $\cyclic(n,3):=\{v_1,\dots v_n\}$ and $\cocyclic(n,3):=\{(-1)^i v_i\colon 1\leq i\leq n\}$ the \emph{cyclic} and \emph{cocyclic} configurations of $n$ vectors in $\R^3$.
 (the combinatorial types of these configurations are independent of the choice of the $t_i$). 
 The configuration $\cyclic(n,3)$ is pointed and corresponds to a point set in $\R^2$ in convex position (points on a parabola); it follows that $f^*_{s,0}(\cyclic(n,3))=f^*_{s,1}(\cyclic(n,3))=0$ for all $s$. Moreover, it is easy to show \cite[Lemma~3.5.1]{Matousek:2003aa} that 
$\cocyclic(n,3)$ is coneighborly.
\end{example}
We are now ready to define the $g$-matrix and the $g^*$-matrix of a vector configuration.
\begin{definition}[$g$-matrix and $g^*$-matrix]
\label{def:g-matrix}
Let $V$ be a configuration of $n$ vectors in $\mathbb{R}^3$. Set 
\[
g_{j,k} (V) := g_{j,k} (\cocyclic(n,3) \to V), \qquad\textrm{and}\qquad g^*_{j,k} (V) := g_{j,k} (V \to \cyclic(n,3))
\]
for $0\leq j\leq r$ and $0\leq k\leq n-r$. We call $g(V)=[g_{j,k}(V)]$ and $g^*(V)=[g^*_{j,k}(V)]$ the \emph{$g$-matrix} and the \emph{$g^*$-matrix} of $V$, respectively.
\end{definition}

Using skew-symmetry, Cor.~\ref{cor:f-g-fstar-gstar-coefficients}, and the properties of cyclic configurations, we get 
\begin{lemma} 
\label{lem:fstar-gstar-b1}
Let $V\in \R^{3\times n}$ be a vector configuration in general position and let $s\geq 4$. 
\[
f_{s,0}^*(V) = \sum_{k=0}^{\floor{\frac{n-4}{2}}}\underbrace{\left( \binom{n-3-k}{s-3}-\binom{k}{s-3}\right)}_{\geq 0}  g^*_{0,k}(V)
\]
and
\[
f^*_{s,\leq 1}(V) = \sum_{k=0}^{\floor{\frac{n-4}{2}}} \underbrace{\textstyle \left( \binom{n-3-k}{s-3}-\binom{k}{s-3}\right)}_{\geq 0}  g^*_{1,k}(V) +  \sum_{k=0}^{\floor{\frac{n-4}{2}}} \underbrace{\left( \binom{n-3-k}{s-4} k - \binom{k}{s-4} (n-3-k) \right)}_{\geq 0}  g^*_{0,k}(V) 
\]
\end{lemma}
The numbers $g^*_{0,k}$, $0\leq k\leq n-3$, correspond to a Gale dual version of the $g$-vector of convex polytopes studied by Lee~\cite{Lee:1991aa} and Welzl~\cite{Welzl:2001aa}. The following is implicit in 
 \cite[Sec.~4]{Welzl:2001aa}:
 \begin{theorem} $V \subset \R^3$ be a configuration of $n$ vectors in general position. Then
 \label{thm:g-star-welzl}
\[g^*_{0,k}(V) \leq  \binom{k+2}{2}\qquad (0\leq k\leq \floor{\frac{n-4}{2}})\]
Equality holds if $V$ is coneighborly.
\end{theorem}
To prove Theorem~\ref{thm:GUBT-fstar} (from which Theorem~\ref{thm:crossings} follows), we need two additional new results.

\begin{lemma}
\label{lem:g-gstar-j1}
Let $V \subset \R^3$ be a configuration of $n$ vectors in general position. Then
\[g_{1,k}(V)+ g^*_{1,k}(V)=(k+1)n-3\binom{k+2}{2} \qquad (0\leq k\leq \floor{\frac{n-4}{2}})\]
\end{lemma}
Our main result is the following:
\begin{theorem} 
\label{thm:g1-rank3}
Let $V \subset \R^3$ be a configuration of $n$ vectors in general position. Then $g_{1,k}(V) \geq 0$ for $0\leq k\leq \floor{\frac{n-4}{2}}$; equivalently (by Lemma~\ref{lem:g-gstar-j1}),
\[g^*_{1,k}(V) \leq (k+1)n-3\binom{k+2}{2} \qquad (0\leq k\leq \floor{\frac{n-4}{2}} )\]
Equality holds if $V$ is coneighborly.
\end{theorem}
\begin{proof}[Proof of Theorem~\ref{thm:GUBT-fstar}]
By Lemma~\ref{lem:fstar-gstar-b1}, for $s\geq 4$, $f_{s,0}^*(V)$ and $f_{s,\leq 1}^*(V)$ are non-negative linear combinations of the numbers $g^*_{0,k}(V)$ and $g^*_{1,k}(V)$, $0\leq k \leq \floor{\frac{n-4}{2}}$. Hence, Theorem~\ref{thm:GUBT-fstar} follows from Theorems~\ref{thm:g-star-welzl} and \ref{thm:g1-rank3}. 
\end{proof}
\begin{remark} The definition of the $g$-matrix generalizes to vector configurations $V\subset \R^r$, for arbitrary $r=d+1 \geq 1$, see the paper \cite{SW:DS} by Streltsova and Wagner, who use this to determine all linear relations between face numbers of levels in arrangements in $S^d$. We conjecture that the small $g$-matrix of any vector configuration in $\R^r$ is nonnegative. This would generalize the \emph{Generalized Lower Bound Theorem} for polytopes and imply the Eckhoff--Linhart--Welzl conjecture in full generality.
\end{remark}

\section{Bounding the Spherical Arc Crossing Number}
\label{sec:proof-crossing}
\begin{observation}
\label{obs:crossings-lin-dep}
Let $V \subset \R^3$ be a configuration of $n$ vectors in general position. Then
\[
f^*_{4,0}(V) + f^*_{4,1}(V) + \frac{1}{2} f^*_{4,2}(V) = \binom{n}{4}
\]
If all vectors in $V$ have unit length (which we can achieve by positive rescaling) then the number of crossings in the induced spherical arc drawing of $K_n$ equals
$\cross(V)= \frac{1}{2}f^*_{4,2}(V)$.
\end{observation}

\begin{proof}
By general position, for every quadruple $W=\{w_1,w_2,w_3,w_4\}\subset V$,  the space of linear dependencies of $W$ has dimension $1$, and hence $\mathcal{F}^*(W)$ consists of exactly two sign vectors $\{F,-F\}$. Thus, there are exactly three combinatorial types of quadruples (see Fig.~\ref{fig:three-types}):
\begin{enumerate}[\textbf{Type~\arabic{enumi}:}]
\setcounter{enumi}{-1}
\item $f^*_{4,0}(W)=1$ (geometrically, $W$ forms the vertex set of a tetrahedron containing the origin in its interior);
\item $f^*_{4,1}(W)=1$ ($W$ spans a pointed cone with three extremal rays that contains the fourth vector in its interior); 
\item $f^*_{4,2}(W)=2$ ($W$ spans a pointed cone with four extremal rays).
\end{enumerate} 
Every quadruple of Type~2 contributes exactly one crossing, other types do not.
\begin{figure}
\begin{center}
\includegraphics[scale=0.7]{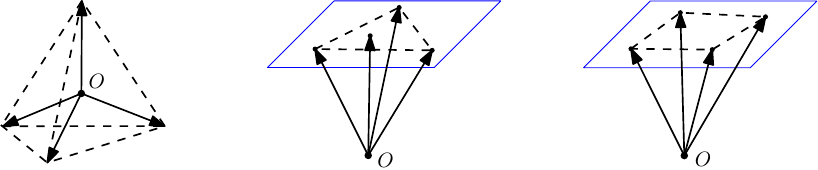}
\end{center}
\caption{The three combinatorial types of four vectors in general position in $\R^3$.\label{fig:three-types}}
\end{figure}
\end{proof}
\begin{proof}[Proof of Theorem~\ref{thm:crossings}]
By Observation~\ref{obs:crossings-lin-dep}, Theorem~\ref{thm:crossings} is equivalent to the statement that 
$f^*_{4,2}(V) \geq 2X(n)$, equivalently, $f^*_{4,0}(V)+f^*_{4,1}(V) \leq \binom{n}{4}-X(n)$, where equality holds in both bounds if $V$ is coneighborly. By the special case $s=4$ of Theorem~\ref{thm:GUBT-fstar},
\begin{equation}
\label{eq:f22-star-Xn}
f^*_{4,\leq 1}(V) \leq \sum_{k=0}^{\floor{\frac{n-4}{2}}} (n-3-2k)\left((k+1)n-3\binom{k+2}{2}\right)=:Y(n),
\end{equation}
with equality if $V$ is coneighborly. A straightforward calculation (see Appendix~\ref{sec:appendix-calculation}) shows that $Y(n)=\binom{n}{4}-X(n)$.
\end{proof}

\begin{remark} 
\label{rem:Wendel}
We mention a connection to geometric probability. Let $\mu$ be a probability distribution on $\R^3$; we assume that $\mu$ is non-degenerate  in the sense that every plane through the origin has $\mu$-measure zero. A beautiful argument due to Wendel~\cite{Wendel:1962aa} shows that if $\mu$ is \emph{centrally symmetric} and $W=\{w_1,w_2,w_3,w_4\} \subset \R^3$ is a set of four independent $\mu$-random vectors then the probability that $W$ is of Type~$i \in \{0,1,2\}$ equals $q_i$, where  $q_0=\frac{\binom{4}{0}+\binom{4}{4}}{2^4}=\frac{1}{8}$, $q_1=\frac{\binom{4}{1}+\binom{4}{3}}{2^4}=\frac{1}{2}$, and $q_2=\frac{\binom{4}{2}}{2^4}=\frac{3}{8}$;
%\footnote{The crux of Wendel's argument is the following: For any \emph{fixed} quadruple $W=\{w_1,w_2,w_3,w_4\}\subset \R^3$ in general position and $\boldsymbol{\varepsilon}=(\varepsilon_1,\varepsilon_2,\varepsilon_3,\varepsilon_4)\in \{+1,-1\}^4$ chosen uniformly at random, $q_i$ equals the probability that the \emph{randomly reorientation} $\boldsymbol{\varepsilon}W:= \{\varepsilon_1 w_1,\varepsilon_2 w_2,\varepsilon_3 w_3,\varepsilon_4 w_4\}$ is of Type~$i$. The result follows by observing that for a $\mu$-random vector $v\in \R^3$ and $\varepsilon \in \{0,1\}$ chosen uniformly at random (and independently of $v$), the vector $\varepsilon v$ is again a $\mu$-random vector, provided $\mu$ is centrally symmetric.}
 it follows that for any set $V\subset \R^3$ of $n$ independent $\mu$-random vectors, the expected number of quadruples of type $i$ equals $q_i \binom{n}{4}$.
 In particular, if the vectors in $V$ are chosen independently and uniformly at random from $S^2$ then the expected number of crossings in the spherical arc drawing given by $V$ is $\frac{3}{8}\binom{n}{4}$, which was also independently shown by Moon~\cite{Moon:1965aa}. Theorem~\ref{thm:crossings}, together with a well-known limiting argument \cite{Scheinerman:1994aa} implies that for an arbitrary (not necessarily centrally symmetric) non-degenerate probability distribution $\mu$ on $S^2$, the expected number of crossings in the spherical arc drawing given by $n$ independent 
 $\mu$-random points is \emph{at least} $\frac{3}{8}\binom{n}{4}$. 
\end{remark}

\section{Gale Duality and Exact Upper Bounds for Sublevels}
\label{sec:Gale}

The definitions of the dissection patterns $\mathcal{F}(V)$, the dependency patterns $\mathcal{F}^*(V)$, the matrices $f(V)$ and $f^*(V)$, and the polynomials $f_V(x,y)$ and $f^*_V(x,y)$ readily generalize to vector configurations $V=\{v_1,\dots,v_n\} \subset \R^r$ in general position and the polar dual spherical arrangements $\arr(V)$ in $S^d$, for $r\geq 1$ and $d=r-1$. In general, 
$f_{s,t}(V)$ counts the $(d-s)$-faces of level $t$ in the simple arrangement $\arr(V)$, and $f^*_{s,t}(V)$ counts the number of linear dependency sign patterns with $t$ negative entries $-1$ and $s$ non-zero entries $\pm 1$.

It is well-known that $\mathcal{F}(V)$ and $\mathcal{F}^*(V)$ determine each other; moreover, the polynomials $f_V(x,y)$ and $f^*_V(x,y)$ also determine each other \cite[Theorem~23]{SW:DS} via the identity
\begin{equation}
\label{eq:fstar-f}
f^*_V(x,y)=(x+y+1)^n - (-1)^r x^n -(x+1)^n f_V(\textstyle -\frac{x}{x+1},\frac{x+y}{x+1})
\end{equation}

 \begin{definition}[Gale Duality] 
\label{def:Gale}
Two vector configurations $V \in \R^{r\times n}$ and $W\in \R^{(n-r)\times n}$ are called \emph{Gale duals} of one another if the rows of $V$ and the rows of $W$ span subspaces of $\R^n$ that are orthogonal complements of one another. Since we always assume that $V$ and $W$ are in general position and of full rank, this is equivalent to the condition $VW^{\top} = 0$.
\end{definition}

Gale dual configurations determine each other up to linear isomorphisms of their ambient spaces $\R^r$ and $\R^{n-r}$, respectively. Thus, we %it is only a mild abuse of terminology to 
speak of \emph{the} Gale dual of $V$, which we denote by $V^*$. Obviously, $(V^*)^*=V$.
Moreover, it follows straightforwardly from the definition that
\[\mathcal{F}^{*}(V)=\mathcal{F}(V^*),\qquad \textrm{hence} \quad f^*_V(x,y)=f_{V^*}(x,y) \quad \textrm{and}\quad f^*_{s,t}(V)=f_{n-s,t}(V^*)\]

The construction from Example~\ref{ex:cyclic-cocyclic} generalizes to yield cyclic configurations $\cyclic(n,r)$ and cocyclic configurations $\cocyclic(n,r)$ for arbitrary $n\geq r\geq 1$ \cite{Ziegler:1995aa,Matousek:2003aa}. These are examples of neighborly and coneighborly configurations, which we define next.

Let $V=\{v_1,\dots,v_n\} \subseteq \R^r$ be a vector configuration in general position. A subset $W\subseteq V$ is \emph{extremal} if there exists an oriented linear hyperplane $H$ that contains all vectors in $W$ and such that one of the two open halfspaces bounded by $H$ contains $V\setminus W$.

\begin{definition}[{Neighborly and Coneighborly Configurations}] A vector configuration 
$V=\{v_1,\dots,v_n\} \subseteq \R^r$ is \emph{coneighborly} if every open linear halfspace contains at least $\lfloor \frac{n-r+1}{2} \rfloor$ vectors of $V$. 
It is \emph{neighborly} if every subset $W \subseteq V$ of size $|W| \leq \lfloor \frac{r-1}{2} \rfloor$ is \emph{extremal}.
\end{definition}
Every neighborly vector configuration $V \subset \R^r$ is pointed, hence corresponds to a point set $S\subset \R^d$, $d=r-1$, and $V$ being neighborly means the simplicial $d$-polytope $\conv(S)$ is a \emph{neighborly polytope}, i.e., every subset of $S$ of size at most $\floor{\frac{d}{2}}$ forms a face. 
We note that for $r=1,2$ ($d=0,1$) neighborliness is the same as being pointed, and for $r=3,4$ ($d=2,3$) $V$ is neighborly iff the point set $S$ is in convex position. By a celebrated result of McMullen~\cite{McMullen:1970aa} neighborly polytopes maximize the number of faces of any dimension:

\begin{theorem}[Upper Bound Theorem for Convex Polytopes] 
\label{thm:UBT}
Let $V \subset \R^r$ be a configuration of $n$ vectors in general position. Then
\[
f_{s,0}(V) \leq f_{s,0}(\cyclic(n,r)) \qquad (0 \leq s \leq d=r-1)
\]
with equality if $V$ is neighborly.
\end{theorem}
Eckhoff \cite[Conj.~9.8]{Eckhoff:1993aa}, Linhart~\cite{Linhart:1994aa}, and Welzl~\cite{Welzl:2001aa}, independently of one another (and in slightly different forms) conjectured a far-reaching generalization of Theorem~\ref{thm:UBT} for sublevels of arrangements. Let $f_{s,\leq k}(V):=\sum_{t\leq k}f_{s,t}(V)$.

\begin{conjecture}[Generalized Upper Bound Conjecture for Sublevels]
\label{GUBC} 
Let $V \subset \R^r$ be a configuration of $n$ vectors in general position, and $0\leq k\leq \floor{\frac{n-r-1}{2}}$. Then 
\[f_{s,\leq k}(V) \leq f_{s,\leq k}(\cyclic(n,r)) \qquad (0 \leq s \leq d=r-1)\]
with equality if $V$ is neighborly.
\end{conjecture}
%We remark that the restriction on $\leq \floor{\frac{n-r-1}{2}}$ is important. E.g., for $r=3$, $n$ even, and $k=\frac{n-2}{2}$, neighborly configurations (which correspond to point sets in convex position $\R^2$) actually \emph{minimize} the number $f_{2,\leq \frac{n-2}{2}}$ of vertices at level $\leq \frac{n-2}{2}$.
A random sampling argument due to Clarkson (see~\cite{Clarkson:1989aa}) shows that Conjecture~\ref{GUBC} is true \emph{asymptotically}, for fixed $r$ and $n,k\to \infty$, up to a constant factor depending on $r$. For the case $s=d$ of vertices at sublevel $(\leq k)$, Conjecture~\ref{GUBC} was shown to be true exactly (without constant factors) by Peck \cite{Peck:1985aa} and by Alon and Gy\H{o}ri~\cite{Alon:1986aa} for $r\leq 3$ and by Welzl~\cite{Welzl:2001aa} for \emph{pointed} vector configurations in rank $r=4$.  Wagner~\cite{Wagner:2006aa} proved that it is true up to a factor of $4$ for arbitrary rank $r$. Explicit formulas for the numbers $f_{s,t}(\cyclic(n,r))$ %(and hence $\varphi^c_{\leq k}(\cyclic(n,r))$) 
have been obtained by Andrzejak and Welzl~\cite{Andrzejak:2003aa}. 

By Gale duality, Theorem~\ref{thm:GUBT-fstar} confirms Conjecture~\ref{GUBC} for $n=r+3=d+4$. 

\begin{remark} Bounding the maximum number of faces at level \emph{exactly} $k$ is of a rather different flavor. For coneighborly configurations, all $2\binom{n}{r-1}$ vertices of the dual arrangement are concentrated at one or two consecutive levels $k=\floor{\frac{n-r+1}{2}}$ and $k=\ceil{\frac{n-r+1}{2}}$. By contrast, determining the maximum number $f_{d,k}$ of vertices at level $k$ for \emph{pointed} vector configurations in $\R^r$ is a difficult open problem, first studied by Lov\'asz~\cite{Lovasz:1971aa} and Erd\H{o}s, Lov\'asz, Simmons, and Straus~\cite{Erdos:1973aa} in the 1970s (see \cite{Wagner:2008aa} or \cite[Ch.~11]{Matousek:2002aa} for more details and background); even for $r=3$ (i.e., $d=2$) there remains a big gap between the best upper and lower bounds to date, which are $O(nk^{1/3})$ (due to Dey~\cite{Dey:1998aa}) and $ne^{\Omega(\sqrt{\log k})}$ (due to T\'oth~\cite{Toth:2001aa}), respectively.
\end{remark}

\section{Continuous Motion and the $g$-Matrix}
\label{sec:g-matrix}

Any two configurations $V=\{v_1,\dots,v_n\}$ and $W=\{w_1,\dots,w_n\}$ of $n$ vectors in general position in $\R^3$ can be deformed into one another through a continuous family $V(t)=\{v_1(t),\dots,v_n(t)\}$ of vector configurations, where $v_i(t)$ describes a continuous path from $v_i(0)=v_i$ to $v_i(1)=w_i$ in $\R^3$. If we choose this continuous motion sufficiently generically, then there be a finite set of events $0 < t_1 < \dots < t_N < 1$
during which the combinatorial type of $V(t)$ (encoded by $\mathcal{F}(V(t))$) changes in a controlled way, by a \emph{mutation}, as described in Section~\ref{sec:g-matrix-intro}. Thus, any two vector configurations are connected by a finite sequence $V=V_0,V_1,\dots,V_N=W$ such that $V_{s-1}$ and $V_s$ differ by a mutation, $1\leq s\leq N$. 

We describe the change from $\mathcal{F}(V)$ to $\mathcal{F}(W)$ when $V$ and $W$ differ by a single mutation. 
Let $R \in \binom{[n]}{3}$ index the unique triple of vectors that become linearly dependent during the mutation; the corresponding triple of great circles in $S^2$ intersect in a pair of antipodal points during the mutation. As discussed above, immediately before and immediately after the mutation, these three great circles bound an antipodal pair of triangular faces $\sigma,-\sigma$ in $\mathcal{A}(V)$ 
(the \emph{disappearing} triangles) and a corresponding pair of triangular faces $\tau,-\tau$ in $\mathcal{A}(W)$ (the \emph{appearing} triangles), respectively, see Figures~\ref{fig:0k3k} and \ref{fig:1k2k}. 

Recall that the \emph{type} of $\tau$ is defined as the pair $(j,k)$, where $j=|R\cap Y_-|$ and $k=|([n]\setminus R)\cap Y_-|$ and $Y\in \mathcal{F}(W)$ is the signature of $\tau$,
The types of $\sigma$, $-\tau$ and $-\sigma$ are $(3-j,k)$, $(3-j,n-3-k)$ and $(j,n-3-k)$, respectively, see Figures~\ref{fig:0k3k} and \ref{fig:1k2k}.
We say that the mutation $V\to W$ is of \emph{type~$(j,k)\equiv(3-j,n-3-k)$} (the pair of types of the triangles $\tau,-\tau$ that appear).

We have $F\in \mathcal{F}(V)\setminus \mathcal{F}(W)$ iff $F$ corresponds to a face of $\arr(V)$ contained in one of the disappearing triangular cells $\sigma,-\sigma$, and $F\in \mathcal{F}(W)\setminus \mathcal{F}(V)$ iff $F$ corresponds to a face of $\arr(W)$ contained in one of the appearing triangular cells $\tau,-\tau$. All other faces are preserved, i.e., they belong to $\mathcal{F}(V) \cap \mathcal{F}(W)$. 
Let us define $f_\sigma(x,y):=\sum_{F\subseteq \sigma}  x^{|F_0|}y^{|F_-|}$, 
where we use the notation $F\subseteq \sigma$ to indicate that the sum ranges over all $F \in \mathcal{F}(V)$ corresponding to faces of $\arr(V)$ contained in $\sigma$. The polynomials $f_{-\sigma}(x,y)$,  $f_\tau(x,y)$, and $f_{-\tau}(x,y)$ are defined analogously. 
These four polynomials have a simple form:
\begin{eqnarray*}
& f_\sigma (x,y)=y^k\left[(x+1)^{j}(x+y)^{3-j}-x^3\right],\quad f_{-\sigma}(x,y)=y^{n-3-k}\left[(x+1)^{3-j}(x+y)^{j}-x^3\right] & \\
& f_\tau(x,y)=y^k\left[(x+1)^{3-j}(x+y)^{j}-x^3\right],\quad f_{-\tau}(x,y)=y^{n-3-k}\left[(x+1)^{j}(x+y)^{3-j}-x^r\right]
\end{eqnarray*}
Thus, we get the following:
\begin{lemma}
\label{lem:f-mutation}
Let $V\to W$ be a mutation of Type~$(j,k)\equiv(3-j,n-3-k)$. Then
\begin{equation}
\label{eq:f-g-mutation}
f_W(x,y)-f_V(x,y) = \left(y^k - y^{n-3-k}\right)\left[(x+1)^{3-j}(x+y)^{j} - (x+1)^{j}(x+y)^{3-j}\right]
\end{equation}
\end{lemma}
As remarked in Sec.~\ref{sec:Gale} above, the polynomials $f_V(x,y)$ and $f^*_V(x,y)$ determine each other via the identity \eqref{eq:fstar-f}. This implies the following (which can also be proven by a direct analysis of how $\mathcal{F}^*$ and hence $f^*$ change during a mutation):
\begin{lemma}
\label{lem:fstar-mutation}
Let $V\to W$ be a mutation of Type~$(j,k)\equiv(3-j,n-3-k)$. Then
\begin{equation}
\label{eq:fstar-g-mutation}
f^*_W(x,y)-f^*_V(x,y) = \left(y^j - y^{3-j}\right)\left[(x+1)^{k}(x+y)^{n-3-k} - (x+1)^{n-3-k}(x+y)^{k}\right]
\end{equation}
\end{lemma}

We are now ready to define the $g$-matrix $g(V\rightarrow W)$ of a pair of vector configurations. 
\begin{definition}[$g$-Matrix of a pair]
\label{def:g-pair}
Let $V,W$ be configurations of $n$ vectors in $\R^3$. 

If $V\to W$ is a single mutation of Type~$(i,\ell)\equiv(r-i,n-3-\ell)$ then we define the $g$-matrix 
$g(V\rightarrow W)=[g_{j,k}(V\rightarrow W)]$, $0\leq j\leq 3$ and $0\leq k \leq n-3$, as follows: 

If 
$2\ell=n-3$, then $g_{j,k}(V\to W)=0$ for all $j,k$. If 
$2\ell\neq n-3$, then
\[
g_{j,k}(V\to W):= 
\begin{cases}
  +1 & \textrm{if } (j,k)=(i,\ell) \textrm{ or }(j,k)=(3-i,n-3-\ell)\\
  -1 & \textrm{if } (j,k)=(3-i,\ell) \textrm{ or }(j,k)=(i,n-3-\ell)\\
  0 & \textrm{else.}  
\end{cases}
\]
More generally, if $V$ and $W$ are connected by a sequence $V=V_0, V_1,\dots, V_N=W$, where $V_{s-1}$ and $V_s$ differ by a single mutation, then we define
\[
g_{j,k}(V\to W):= \sum_{s=1}^N g_{j,k}(V_{s-1}\to V_s)
\]
\end{definition}

\begin{proof}[Proof of Theorem~\ref{thm:f-g}]
The skew-symmetries of the $g$-matrix and the fact that $g(V\to W)=-g(W\to V)$ and $g(U\to W)=g(U\to V)+g(V\to W)$ follow immediately from the definition. Moreover, \eqref{eq:f-poly-g-poly} and \eqref{eq:f*-g*} follow from Lemmas~\ref {lem:f-mutation} and \ref{lem:fstar-mutation}.
\end{proof}

By Gale duality, Theorem~\ref{thm:f-g} implies:
\begin{corollary} Let $V,W\in \R^{r\times n}$ be vector configurations, and let $V^*,W^*\in  \R^{(n-r)\times n}$ be their Gale duals. Then
\label{cor:g-pair-dual}
\[ g_{j,k}(V \to W) = -g_{k,j}(V^* \to W^*).\]
\end{corollary}

\section{The Upper Bound for $g_{1,k}^*$ in Rank $3$}
\label{sec:g1-arcs}
In this section, we prove Lemma~\ref{lem:g-gstar-j1}  and Theorem~\ref{thm:g1-rank3}.
Let $V=\{v_1,\dots,v_n\}\subset \R^3$ be a configuration of $n$ vectors in general position.  It will be convenient to view $V$ as the set $S-o=\{v_i=p_i-o\colon 1\leq i\leq n\}$ of differences between a point set $S=\{p_1,\dots,p_n\} \subset \R^3$ (the \emph{tips} of the vectors) and a fixed point $o\in \R^3$ (the \emph{origin}). 

Choose a line $\ell$ in $\R^3$ through $o$ in general position (in particular, $\ell$ is not contained in the linear span of any pair of vectors in $V$). By positively rescaling the vectors $v_i$, 
we may assume that $S$ is a subset of the cylinder  $Z$ consisting of the points in $\R^3$ at Euclidean distance exactly $1$ from $\ell$. It follows that $S$ is a point set in convex position.

In what follows, we will consider vector configurations of the form $S-p_0$, where $p_0\in \ell$, and how these configurations change as we move $p_0$ continuously along $\ell$. 
Given $S=\{p_1,\dots,p_n\}$ and $p_0\in \ell$, define $w_i:=(1,p_i) \in \{1\}\times \R^3 \subset \R^4$, $u_i:=\frac{1}{\|w_i\|}w_i \in S^3$, $0\leq i\leq n$, and $U:=\{u_1,\dots, u_n\}$. Let $U/u_0$ be the orthogonal projection of $U$ onto the hyperplane $u_0^\bot$ (viewed as a vector configuration in $u_0^\bot\cong \R^3$). Note that $U/u_0$ has the same combinatorial type as $S-p_0$. In particular, for $p_0=o$, $U/u_0$ has the same combinatorial type as $V=S-o$.
As $p_0$ moves along $\ell$, the corresponding vector $u_0 \in S^3$ moves along an open great semi-circle in $S^3$ (the radial projection of the line $\{1\}\times \ell$ onto $S^3$), whose closure is a closed great semi-circle $\gamma$ in $S^3$ that connects two antipodal points $u_{-\infty}$ and $u_{+\infty}$ (corresponding to ``points on $\ell$ at $\pm \infty$''). By general position of $\ell$, $\gamma$ is in general position with respect to $U$.

\begin{lemma} 
\label{lem:special}
The pair consisting of the vector configuration $U=\{u_1,\dots,u_n\}$ and the semi-circle $\gamma\subset S^3$ has the following properties:
\begin{enumerate}
\item The rank $4$ vector configuration $U \subset S^3 \subset \R^4$ is neighborly.
\item The rank $3$ vector configuration $U/u_{-\infty}=U/u_{+\infty}$ is neighborly.
\item Let $u_0 \in \gamma \subset S^3$ and $1\leq i \leq n$. Then the vector $u_i$ is extremal in $U \sqcup \{ u_0 \}$ (i.e., there exists $w\in S^3$ such that $\langle w,u_i\rangle =0$ and $\langle w,u_s\rangle > 0$ for $s\in \{0\}\sqcup ([n]\setminus \{i\})$). 
\end{enumerate}
\end{lemma}
\begin{proof} Since the point set $S\subset \R^3$ is in convex position (by virtue of being contained in the cylinder $Z$), the vector configuration $U:=\{u_1,\dots,u_n\} \subset S^3 \subset \R^4$ is neighborly. 

Moreover, linear hyperplanes in $\R^4$ orthogonal to $u_{\pm \infty}$ correspond to affine planes in $\R^3$ that are parallel to the line $\ell$, or equivalently to affine lines in the plane $\ell^\bot$ orthogonal to $\ell$. Thus, the vector configuration $U/u_{-\infty}=U/u_{+\infty}$ is neighborly because the orthogonal projection of $S$ to $\ell^\bot\cong \R^2$ is in convex position. 

Finally, the third property follows because for every point $p_i \in S$, the tangent plane to the cylinder $Z$ at $p_i$ contains no other point of $P$, by general position, hence $\{p_0\}\sqcup P\setminus\{p_i\}$ is contained in one of the open affine halfspaces determined by that tangent plane.
\end{proof}
For $0\leq i \leq n$, let $H_i:=\{x\in S^3\colon \langle u_i,x\rangle =0\}$, and let $\mathcal{A}(U)$ be the arrangement of hemispheres $H_1^+,\dots, H_n^+$ in $S^3$. Note that the rank $3$ vector configuration $U/u_0$ is polar dual to the $2$-dimensional arrangement $\mathcal{A}(U)\cap H_0$ in $H_0 \cong S^2$.

\begin{definition}[$k$-Arcs, $\Lambda_k(U)$, $\lambda_k(U,u_0)$] 
\label{def:k-arcs}
Let $k\leq \frac{n-4}{2}$. Let $C=H_i\cap H_j$ be the great circle in $S^3$ formed by the intersection of two great $2$-spheres of the arrangement $\mathcal{A}(U)$, $1\leq i<j\leq n$. Consider the subgraph of $\mathcal{A}(U)$ consisting of the vertices and edges of $\mathcal{A}(U)$ contained in $C$ and at sublevel $(\leq k)$. This subgraph cannot cover the entire circle $C$, since $C$ contains at least one pair of antipodal vertices at levels $\floor{\frac{n-2}{2}}, \ceil{\frac{n-2}{2}}$. Thus, the connected components of this subgraph form closed intervals, which we call \emph{$k$-arcs} of $\mathcal{A}(U)$. 

We denote by  $\Lambda_k(U)$ the number of $k$-arcs in $\mathcal{A}(U)$, and we define $\lambda_k(U,u_0)$ as the number of $k$-arcs of $\mathcal{A}(U)$ that are completely contained in the open hemisphere $H_0^-$.
\end{definition}
Our main technical result (which by the preceding discussion implies both Lemma~\ref{lem:g-gstar-j1} and Theorem~\ref{thm:g1-rank3}) is the following. 
\begin{proposition}
\label{prop:lambda-g}
Consider a vector configuration $U=\{u_1,\dots,u_n\} \subset \R^4$ and a great semicircle $\gamma$ in $S^3$ with endpoints $u_{\pm \infty}$ such that the properties of Lemma~\ref{lem:special} are satisfied. 

Let $0\leq k\leq \floor{\frac{n-4}{2}}$,, and let $u_0\in \gamma$ such that $U\sqcup \{u_0\}$ is in general position. Then
\[
g_{1,k}(U/u_0)=\lambda_k(U,u_0)\qquad \textrm{and}\qquad g^*_{1,k}(U/u_0)=(k+1)n-3\binom{k+2}{2}-\lambda_k(U,u_0)
\]
\end{proposition}

To prove Proposition~\ref{prop:lambda-g}, we consider how the vector configuration $U/u_0$ changes as $u_0$ moves continuously along $\gamma$. In terms of the dual arrangement, a mutation of $U/u_0$ occurs exactly when the great $2$-sphere $H_0$ dual to $u_0$ passes over a vertex $v$ of the arrangement 
$\mathcal{A}(U)$, whose level we will denote by $k$ (concurrently, $H_0$ passes over the antipodal vertex $-v$ of level $n-3-k$, so we will assume that $k\leq \floor{\frac{n-3}{2}}$). The vertex $v$ is incident to exactly three edges of $\mathcal{A}(U)$ at level $k$ (each of which corresponds to a $k$-arc with endpoint $v$). Consider the moment right before or right after the mutation when $v$ lies in the negative halfspace $H_0^-$, and let $j\in \{0,1,2,3\}$ denote the number of level $k$ edges of $\mathcal{A}(U)$ incident to $v$ that are contained in the negative hemisphere $H_0^-$. We call this a \emph{transition of Type~$(j,k)^+$ or $(j,k)^-$} depending on whether $v$ lies in $H_0^+$ or in $H_0^-$ after the mutation.

\begin{lemma}
\label{lem:transitions}
Let $0\leq k\leq \floor{\frac{n-4}{2}}$, and suppose that $u_0$ moves continuously along $\gamma$. 
\begin{enumerate}
\item There are no transitions of Type~$(2,k)$ or $(3,k)$.
\item if $(U,u_0)\to (U,u_0')$ is a transition of Type~$(0,k)^{\pm}$ then $\lambda_k$ and $g_{1,k}$ remain unchanged (as well as all other $g_{j,\ell}$, $\ell \neq k$), and $g_{0,k}(U/u_0 \to U,u'_0) = \mp 1$; 
\item if $(U,u_0)\to (U,u_0')$ is a transition of Type~$(1,k)^{\pm}$ then $g_{0,k}$  (as well as all other $g_{j,\ell}$, $\ell \neq k$) remain unchanged, $\lambda_k(U,u'_0)=\lambda_k(U,u_0) \mp 1$, and $g_{0,k}(U/u_0 \to U,u'_0) = \mp 1$.
\end{enumerate}
\end{lemma}

\begin{corollary} 
\label{cor:difflambdak-g1k}
For $0\leq k\leq \floor{\frac{n-4}{2}}$ and pair $u_0,u'_0 \in \gamma$, 
\[
\lambda_k(U,u'_0) - \lambda_k(U,u_0) = g_{1,k}(U/u'_0)- g_{1,k}(U/u_0)= g_{1,k}( (U/u_0) \to (U/u'_0)) 
\]
\end{corollary}

\begin{proof}[Proof of Lemma~\ref{lem:transitions}]
Suppose there is a transition of Type~$(2,k)^{\pm}$ or Type~$(3,k)^{\pm}$. Let $v$ be the corresponding vertex at level $k$ in $\mathcal{A}(U)$ and consider the moment just before or after the transition when $v$ is in the negative hemisphere $H_0^-$. Without loss of generality, let $H_1,H_2,H_3$ be the three great $2$-spheres of $\mathcal{A}(U)$ intersecting in $v$ and in the antipodal vertex $-v$, which is at level $n-3-k$. Since $k\leq \floor{\frac{n-4}{2}}$, we have $k<n-3-k$, hence there is another hemisphere of $\mathcal{A}$, say $H_4$, such that $v\in H_4^+$ and $-v\in H_4^-$.

If the transition is of  Type~$(3,k)^{\pm}$, this implies that the intersection of the hemispheres $\bigcap_{i=0}^4 H_0^+$ is empty, which is impossible since $\{u_0\}\sqcup U$ is a pointed configuration. Suppose then that the transition is of Type~$(2,k)^{\pm}$. Consider the unique level $k$ edge of $\mathcal{A}(U)$ incident to $v$ that is not contained in $H_0^-$. This edge lies in the intersection of two out of the three great $2$-spheres of $\mathcal{A}(U)$ intersecting in $v$, say in the intersection $H_2\cap H_3$. It follows that the intersection $H_0^+\cap H_1^- \cap \bigcap_{i=2}^4 H_i^+$ is empty. This means that there is a linear dependency between the vectors $u_0,u_1,u_2,u_3,u_4$ with exactly one negative coefficient, corresonding to $u_1$. This contradicts Property~3 of Lemma~\ref{lem:special} (one of the defining properties of a special pair).
\begin{figure}[ht]
\begin{center}
\includegraphics[scale=1]{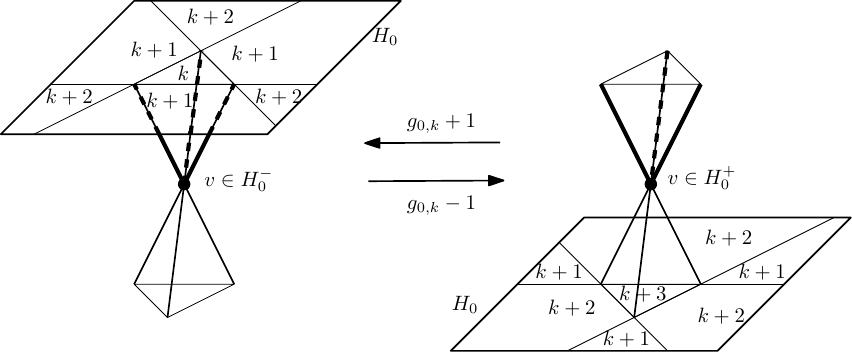}
\end{center}
\caption{Transitions of Type $(0, k)^+$ (left to right) and $(0, k)^-$ (right to left) and the change in $g_{0,k}(U/u_0)$; the labels $k,k+1,k+2,k+3$ show the levels of the $2$-cells in $\mathcal{A}\cap H_0$.}
\label{fig:0klambda}
\end{figure}

\begin{figure}[ht]
\begin{center}
\includegraphics[scale=1]{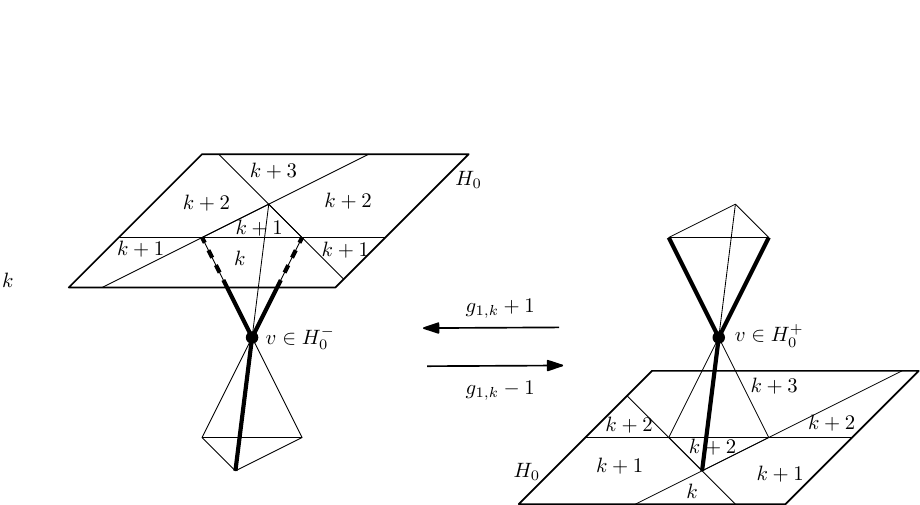}
\end{center}
\caption{Transitions of Type $(1, k)^+$ (left to right) and $(1, k)^-$ (right to left) and the corresponding change in $g_{1,k}(U/u_0)$; the labels $k,k+1,k+2,k+3$ show the levels of the $2$-cells in $\mathcal{A}\cap H_0$.}
\label{fig:1klambda}
\end{figure}

Only the three $k$-arcs ending at $v$ can contribute to a change in $\lambda_k$ during a transition. If the transition is of Type~$(0,k)^{\pm}$, then the three $k$-arcs incident to $v$ intersect $H_0^+$ both before and after the transition, so $\lambda_k$ does not change. If the transition is of Type~$(1,k)^{\pm}$, then there is a unique $k$-arc incident to $v$ that intersects $H_0^-$, and this $k$-arc contributes $1$ to $\lambda_k$ when $v\in H_0^-$ and $0$ otherwise. This shows that $\lambda_k$ changes as stated. 

For the changes in $g_{j,k}(U/u_0)$, $j=0,1$, see Figures~\ref{fig:0klambda} and \ref{fig:1klambda}.\end{proof}
We will need the following well-known fact (see, e.g., \cite[Corollary~8]{Welzl:2001aa}).
\begin{lemma} 
\label{lem:k-facets-neighborly}
Let $W\subset \R^r$ be a neighborly arrangment of $n$ vectors, $r\in \{3,4\}$, and let $\mathcal{A}(W)$ be the polar dual arrangement in $S^d$, $d=r-1$. Then for $0\leq k\leq n-r+1$, the number $f_{d,k}(W)$ of vertices of $\mathcal{A}(W)$ at level $k$ equals $f_{2,k}(W)=n$ for $r=3$, and $f_{3,k}(W)=2n(k+1)-4\binom{k+2}{2}=2(k+1)(n-k-2)$ for $r=4$.
\end{lemma}

\begin{proof}[Proof of Proposition~\ref{prop:lambda-g}]
For every $k$-arc $\alpha$ of  $\mathcal{A}$, the two endpoints of $\alpha$ are vertices of $\mathcal{A}$ at level $k$ (while the interior of a $k$-arc consists of edges at sublevel $(\leq k)$. Conversely, every vertex at level $k$ is an endpoint of exactly three $k$-arcs (corresponding to the three edges at level $k$ incident to that vertex). Thus, By Lemma~\ref{lem:k-facets-neighborly},
\begin{equation}
\label{eq:Lambdak}
\Lambda_k(U)=3(k+1)(n-k-2)=3(k+1)n-6\binom{k+2}{2}
\end{equation}
Consider the great $2$-sphere $H_{\infty}:=\{x\in S^3\colon \langle u_{+\infty},x\rangle =0\}$ and the two open hemispheres $H^+_{\infty}:=\{x\in S^3\colon \langle u_{+\infty},x\rangle > 0\}$
and $H^-_{\infty}:=\{x\in S^3\colon \langle u_{+\infty},x\rangle < 0\}$. The $2$-dimensional arrangement $\mathcal{A}\cap H_{\infty}$ is polar dual to the neighborly configuration $U/u_{+\infty}=U/u_{+\infty}$, hence the number of vertices of $\mathcal{A}\cap H_{\infty}$ at sublevel $(\leq k)$ equals $(k+1)n$.

Consider a $k$-arc $\alpha$ of $\mathcal{A}$. Then either $\alpha \subset H^+_{\infty}$, or $\alpha \subset H^-_{\infty}$, or $\alpha$ intersects $H_\infty$; $k$-arcs of the first two kinds are counted by $\lambda_k(U,u_{-\infty})$ and $\lambda_k(U,u_{+\infty})$, respectively, while $k$-arcs of the third kind correspond bijectively to vertices at sublevel $(\leq k)$ in $\mathcal{A}\cap H_{\infty}$. Therefore,
\[
\Lambda_k(U)=3(k+1)n-6\binom{k+2}{2}=\lambda_k(U,u_{-\infty}) + \lambda_k(U,u_{+\infty}) +(k+1)n
\]
hence
\begin{equation}
\label{eq:Lambdak-lambdak}
\lambda_k(U,u_{-\infty}) + \lambda_k(U,u_{+\infty}) = 2(k+1)n + 6\binom{k+2}{2}
\end{equation}
Consider now a unit vector $u_0 \in \gamma$. By applying Corollary~\ref{cor:difflambdak-g1k} with $u'_0 =u_{+\infty}$ and using that $U/u_{+\infty}$ is neighborly, we get
\begin{equation}
\label{eq:difflambda-diffg1star}
g^*_{1,k}(U/u_0) = g_{1,k}((U/u_0) \to (U/u_{+\infty}))
= \lambda_k(U,u_{+\infty}) - \lambda_k(U,u_0)
\end{equation}
In particular, since $U/u_{+\infty}$ and $U/u_{-\infty}$ are identical, Equations~\eqref{eq:Lambdak-lambdak} and \eqref{eq:difflambda-diffg1star}  imply
\begin{equation}
\label{eq:lambdapminfty}
\lambda_k(U,u_{+\infty}) = \lambda_k(U,u_{-\infty}) = (k+1)n + 3\binom{k+2}{2}
\end{equation}
Substituting this back into \eqref{eq:difflambda-diffg1star}, we get, for every $u_0\in \gamma$,
\begin{equation}
\label{eq:g1k-lambdak}
g^*_{1,k}(U/u_0) = (k+1)n-3\binom{k+2}{2} - \lambda_k(U,u_0)
\end{equation}
which proves the second half of Proposition~\ref{prop:lambda-g}.

It remains to show that $g^*_{1,k}(V)= (k+1)n-3\binom{k+2}{2}$ for $0\leq k\leq \floor{\frac{n-4}{2}}$ whenever $V$ is coneighborly. This follows from the following observation: If  $V=S-o=U/u_0$ is coneighborly, then every vertex $v$ of $\mathcal{A}(U)$ at level $k$ must be contained in the positive hemisphere $H_0^+$ (otherwise, $v$ would correspond to an affine hyperplane $H$ in $\R^3$ that is spanned by three points of $S$ and such that $o\in H^-$ and $|H^-\cap S|\leq k$; but then we could translate $H$ towards $o$ and obtain a hyperplane through $o$ that still contains at most $\floor{\frac{n-4}{2}}= \floor{\frac{n-2}{2}}-1$ points of $S$, contradicting coneighborliness of $V=S-o$). It follows from the observation that all $k$-arcs of $\mathcal{A}(U)$ are contained in $H_0^+$, hence $\lambda_k(U,u_0)=0$, which together with \eqref{eq:g1k-lambdak} implies 
$g^*_{1,k}(V)= (k+1)n-3\binom{k+2}{2}$ as we wanted to show.
\end{proof}

We remark that the notion of $k$-arcs in arrangements in $S^3$ (and in particular Equation~\eqref{eq:Lambdak}) have been used before in related contexts, see, e.g., \cite{Edelsbrunner:1989aa,Ramos:2009aa}. 

\bibliography{g-matrix}

\appendix
\section{The Remaining Calculation for the Spherical Arc Crossing Number}
\label{sec:appendix-calculation}
We want to show that 
\begin{equation}
\sum_{k=0}^{\floor{\frac{n-4}{2}}} (n-3-2k)\left((k+1)n-3\binom{k+2}{2}\right)
=\binom{n}{4}-X(n)
\tag{\ref{eq:f22-star-Xn}}
\end{equation}
where 
\[
X(n)=\frac{1}{4}\Big\lfloor \frac{n}{2}\Big\rfloor \Big\lfloor \frac{n-1}{2} \Big\rfloor \Big\lfloor \frac{n-2}{2} \Big\rfloor \Big\lfloor \frac{n-3}{2}\Big\rfloor
\]
The left-hand side of \eqref{eq:f22-star-Xn} equals
\begin{equation}
\label{eq:sums}
(n-3)n\sum_{k=0}^{\floor{\frac{n-4}{2}}} (k+1) 
- 3(n-3)\sum_{k=0}^{\floor{\frac{n-4}{2}}} \binom{k+2}{2}
- 2n\sum_{k=0}^{\floor{\frac{n-4}{2}}} k(k+1) 
+ 6 \sum_{k=0}^{\floor{\frac{n-4}{2}}} k\binom{k+2}{2}
\end{equation}
Using the basic identity
\[ 
\sum_{k = 0}^m \binom{k+i}{i} = \binom{m+i+1}{i+1} 
\]
we see that \eqref{eq:sums} equals
\begin{eqnarray*}
& \displaystyle n(n-3) \binom{\floor{\frac{n}{2}}}{2} - 3 (n-3) \binom{\floor{\frac{n+2}{2}}}{3} - 4n \binom{\floor{\frac{n}{2}}}{3} + 18 \binom{\floor{\frac{n+2}{2}}}{4} & \\[1ex]
=& \displaystyle \left\lfloor \frac{n}{2} \right\rfloor \left\lfloor \frac{n-2}{2}\right\rfloor \underbrace{\left( \frac{1}{2} n(n-3) - \frac{1}{2} (n-3) \left\lfloor \frac{n+2}{2} \right\rfloor -\frac{2}{3}n \left\lfloor \frac{n-4}{2} \right \rfloor + \frac{3}{4} \left\lfloor \frac{n+2}{2} \right\rfloor \left\lfloor \frac{n-4}{2} \right\rfloor \right)}_{(*)}\\
= &  \displaystyle \left\lfloor \frac{n}{2} \right\rfloor \left\lfloor \frac{n-2}{2}\right\rfloor \underbrace{\left( \frac{2}{3} \left( \left\lfloor \frac{n-1}{2} \right\rfloor+ \frac{1}{2} \right) \left( \left\lfloor \frac{n-3}{2} \right\rfloor+ \frac{1}{2} \right) - \frac{1}{4} \left\lfloor \frac{n-1}{2} \right\rfloor \left\lfloor \frac{n-3}{2} \right\rfloor \right)}_{(**)}
\end{eqnarray*}
where the last step amounts to the assertion that the expressions (*) and (**) are equal as quadratic polynomials in $n$, which is easy to check. Equation~\eqref{eq:f22-star-Xn} then follows from the identify 
\[ 
\binom{n}{4} = \frac{2}{3}\left\lfloor \frac{n}{2} \right\rfloor \left\lfloor \frac{n-2}{2}\right\rfloor \left( \left\lfloor \frac{n-1}{2}\right\rfloor + \frac{1}{2} \right)  \left( \left\lfloor \frac{n-3}{2}\right\rfloor + \frac{1}{2} \right) 
\]
which is likewise easy to verify. \hfill \qed

\end{document}